\documentclass[reqno,12pt, oneside]{amsart}

\usepackage{amssymb}
\usepackage[hidelinks]{hyperref}
\usepackage[textsize=footnotesize]{todonotes}
\usepackage{tikz}
\usetikzlibrary{cd}
\usepackage{float}
\usepackage{xcolor}
\usepackage{todonotes}
\usepackage{framed}

\theoremstyle{plain}
\newtheorem{theorem}{Theorem}[section]
\newtheorem{lemma}[theorem]{Lemma}
\newtheorem{definition}[theorem]{Definition}
\newtheorem{corollary}[theorem]{Corollary}
\newtheorem{proposition}[theorem]{Proposition}
\newtheorem{remark}[theorem]{Remark}
\newtheorem{example}[theorem]{Example}

\DeclareMathOperator{\var}{var}

\newcommand{\R}{\mathbb R}

\newcommand{\N}{\mathbb N}

\newcommand{\I}{\mathcal I}
\newcommand{\J}{\mathcal J}
\newcommand{\cP}{\mathcal P}
\newcommand{\Fin}{\mathrm{Fin}}
\newcommand{\Exh}{\mathrm{Exh}}
\newcommand{\FIN}{\mathrm{FIN}}
\newcommand{\EXH}{\mathrm{EXH}}
\newcommand{\mFIN}{\mathrm{mFIN}}
\newcommand{\mEXH}{\mathrm{mEXH}}

\newcommand{\BV}{\mathrm{BV}}
\newcommand{\Mon}{\mathrm{Mon}}

\newcommand{\norm}[1]{\left\|#1\right\|}

\newcommand{\abs}[1]{\lvert#1\rvert}

\newcommand{\zdef}{{\mathrel{\mathop:}=}}

\newcommand{\lnorm}[1]{\left\|#1\right\|_{\Lambda BV}}

\newcommand{\ABV}{{ABV}}
\newcommand{\BBV}{{BBV}}

\title{Functions of bounded variation from ideal perspective}

\author[J.~Gulgowski]{Jacek Gulgowski}
\address[J.~Gulgowski]{Institute of Mathematics\\ Faculty of Mathematics, Physics and Informatics\\ University of Gda\'{n}sk\\ ul. Wita Stwosza 57\\ 80-308 Gda\'{n}sk\\ Poland}
\email{Jacek.Gulgowski@ug.edu.pl}

\author[A.~Kwela]{Adam Kwela}
\address[A.~Kwela]{Institute of Mathematics\\ Faculty of Mathematics\\ Physics and Informatics\\ University of Gda\'{n}sk\\ ul.~Wita  Stwosza 57\\ 80-308 Gda\'{n}sk\\ Poland}
\email{Adam.Kwela@ug.edu.pl}
\urladdr{https://mat.ug.edu.pl/~akwela}

\author[J.~Tryba]{Jacek Tryba}
\address[J.~Tryba]{Institute of Mathematics\\ Faculty of Mathematics, Physics and Informatics\\ University of Gda\'{n}sk\\ ul. Wita Stwosza 57\\ 80-308 Gda\'{n}sk\\ Poland}
\email{Jacek.Tryba@ug.edu.pl}

\begin{document}

\begin{abstract}
We present a unified approach to two classes of Banach spaces defined with the aid of variations: Waterman spaces and Chanturia classes. Our method is based on some ideas coming from the theory of ideals on the set of natural numbers.
\end{abstract}

\maketitle

\section{Introduction}

The concept of the variation of the function was introduced in 1881 by Camille Jordan and found plenty of applications and generalizations since that time. When we look closely at the definition of the variation of the function we can see its tight relation to the question of (un)boundedness of the series. This becomes even more evident, when we look at some of the generalizations of Jordan's definition, namely the $\Lambda$-variation introduced by Waterman in 1972 (see \cite{Wat1}) and Chanturia classes introduced in 1974 (see \cite{Cha74}).

The question of (un)boundedness of the series of real numbers naturally appears also in the studies of the ideals on the set of natural numbers, with the paper \cite{FKT} as the very recent example of this perspective and the concept of the summable ideal (introduced below), which is a very basic notion in the theory of ideals on the set of natural numbers.

Looking at these two separate threads in the realm of mathematics it appeared to be very appealing to us to join them: to look at different spaces of functions of bounded variation from the perspective of the theory of ideals defined on the set of natural numbers. The additional inspiration arrived from the recent paper \cite{FBN} by Borodulin-Nadzieja and Farkas, who showed that the concept of ideals introduced by l.s.c. submeasure on $\N$ (see definitions below) naturally defines certain Banach sequence spaces. On the other hand, these sequence spaces may be used as a natural foundation for the definition of spaces of functions of some type of bounded variation, with a concept of variation generalizing many different attitudes (especially Waterman variation and Chanturia classes).

In this paper we define the concept of the variation of the function defined on a compact real interval originating from the l.s.c submeasure $\phi$ defined on $\N$ (this is laid out in Section 3). In the Section 4 we study the inclusions between different spaces and corresponding relations between ideals generated by the submeasure, in a general setting. Then, in Section 5 we show that for simple density ideals we recreate the Chanturia classes; while in Section 6 we show that summable ideals correspond directly to the concept of Waterman $\Lambda$-variation. It appears that in these two special cases the inclusions between the spaces of functions of bounded variation may be nicely described by the relation between corresponding ideals (in terms of inclusion or Kat\v{e}tov order). One of these results also leads us to a new characterization of Kat\v{e}tov order between summable ideals.

\section{Preliminaries}

\subsection{Basics about sequence spaces} 

By $\R^\N$ we will denote the family of all real-valued sequences.

We will refer to several standard Banach sequence spaces. Here we will present a notation, and basic properties, which will be used in the sequel. First of all, in the examples listed below we set a sequence $x=(x_n)_{n\in\N}\in\R^\N$ of real numbers.
\begin{itemize}
\item $\ell_\infty$ denotes the space of all bounded sequences equipped with the supremum norm $\norm{x}_\infty = \sup_{n\in \N}|x_n|$;
\item $c_0$ denotes the subspace of $l_\infty$ consisting of all sequences such that $\lim_{n\to+\infty} x_n = 0$.
\item $\ell_1$ denotes the space of such sequences that $\norm{x}_{\ell_1}=\sum_{n\in\N}|x_n|<+\infty$.
\end{itemize}

\subsection{Basics about \texorpdfstring{$\Lambda BV$}{} spaces}

Let us assume that $A=(a_n)_{n\in\N}$ is such nonincreasing sequence of positive real numbers that $\sum_{n=1}^{\infty} a_n = +\infty$. We call such sequence \emph{a Waterman sequence}. If additionally $\lim_{n\to +\infty}a_n = 0$, we say that the sequence $A$ is a \emph{proper Waterman sequence}.

\begin{remark}
In many sources the Waterman sequence is defined in a form $(\frac{1}{\lambda_n})_{n\in\N}$ where $(\lambda_n)_{n\in\N}$ is nondecreasing and such that    $\sum_{n=1}^{\infty} \frac{1}{\lambda_n} = +\infty$. Then if $\lim_{n\to+\infty} \lambda_n = +\infty$ we have a proper Waterman sequence. Of course by putting $a_n = \frac{1}{\lambda_n}$ we can see that the two definitions are essentially identical.
\end{remark}

Let us denote the unit interval by $I=[0,1]$. Moreover, by ${\mathcal P}_I$ we denote the set of all sequences  of nonoverlapping, closed  subintervals $\{I_1, I_2, \ldots, I_N, \ldots\}$ of $I$. The intervals may be {\em degenerate}, i.e. it may happen that $I_n$ consists of only one point. 

 \begin{definition}
  Let $A = (a_n)_{n\in\mathbb{N}}$ be a Waterman sequence and let $x\colon I\to\mathbb{R}$. We say that $x$ is of bounded $A$-variation  if there exists a positive  constant $M$ such that for any sequence of nonoverlapping  subintervals $\{I_1, I_2, I_3, \ldots, \} \in {\mathcal P}_I$, the following inequality holds
  \[
   \sum_{n=1}^{+\infty} a_n\abs{x(I_n)} \leq M,
  \]
  where $I_n = [s_n,t_n]$ and $|x(I_n)| = |x(t_n)-x(s_n)|$.
  The supremum of the above sums, taken over the family ${\mathcal P}_I$ of all sequences of nonoverlapping subintervals of $I$, is called the $A$-variation of $x$  and it is denoted by $\var_A(x)$. 
 \end{definition}

 \begin{remark}
 The special case of a sequence constantly equal to $1$ corresponds to the classical Jordan variation of a function $x$, which will later be denoted by $\var(x)$.
 \end{remark}
 
This concept was introduced by Waterman in \cite{Wat1}. Since then the functions of bounded $A$-variation were intensively studied by many authors -- for an overview we refer to \cite{ABM}.
 
It is worth to mention that there are many equivalent ways to express that the function $x\colon I\to\R$ is of bounded $A$-variation (cf. \cite[Theorem 1, p.~34]{Wat1976} and \cite[Proposition 1]{BCGS}), but we will not go into the details here.

The space of all  functions defined on the interval $I$ and  of bounded $A$-variation, endowed with the norm $\lnorm{x} \zdef \abs{x(0)} + \var_A(x)$ forms a Banach space $\ABV(I)$ (see \cite[Section 3]{Wat1976}).

The spaces $\ABV(I)$ are proper subspaces of the space $B(I)$ of all bounded functions $x\colon I\to\R$. The space $B(I)$ is equipped with the standard supremum norm
\[
\norm{x}_\infty = \sup_{t\in I} |x(t)|.
\]

\subsection{Basics about ideals}

 \begin{definition}
 A family $\I\subseteq\cP(\N)$ is called an \emph{ideal} if
 \begin{itemize}
    \item $\N\notin\I$,
    \item if $F\subseteq\N$ is finite, then $F\in\I$,
    \item if $C\in\I$ and $D\subseteq C$, then $D\in\I$,
    \item if $C,D\in\I$, then $C\cup D\in\I$.
 \end{itemize}
 \end{definition}

  \begin{definition}
  \label{def-summable}
An ideal is called a \emph{summable ideal} if it is of the form 
$$\I_A=\left\{C\subseteq\N: \sum_{n\in C}a_n<\infty\right\},$$
for some sequence of positive real numbers $A=(a_n)_{n\in\N}$ such that $\sum_{n=1}^\infty a_n=\infty$.
 \end{definition}

 \begin{remark}
    Note that in the definition of summable ideals we do not require that $A$ is nonincreasing. However, in our paper we will only consider summable ideals given by nonincreasing sequences.
 \end{remark}

 By $\Fin$ we denote the smallest ideal, i.e., the one consisting only of all finite subsets of $\N$. Note that $\Fin$ is a summable ideal (given by the sequence $(a_n)$ constantly equal to $1$).

  \begin{definition}
If $\I$ and $\J$ are ideals then we say that \emph{$\I$ is below $\J$ in the Kat\v{e}tov order} and write $\I\leq_K\J$ whenever there is a function $f:\N\to\N$ such that $f^{-1}[C]\in\J$ for every $C\in\I$.
 \end{definition}

Note that actually, despite its name, Kat\v{e}tov order is only a pre-order, not a partial order (it is not antisymetric). Kat\v{e}tov order was introduced in the 1970s in papers \cite{Kat2} and \cite{Kat1} by M. Kat\v{e}tov.

We say that an ideal $\I$ is \emph{tall} if for every infinite $C\subseteq\N$ there is an infinite $D\subseteq C$ such that $D\in\I$. It is easy to see that $\I$ is not tall if and only if  $\I\leq_K\Fin$. Consequently, all non-tall ideals are $\leq_K$-equivalent (i.e., $\I\leq_K\J$ and $\J\leq_K\I$ for any two non-tall ideals $\I$ and $\J$). If $A=(a_n)_{n\in\N}$
 is a sequence of positive real numbers such that $\sum_{n=1}^\infty a_n=\infty$, then $\I_A$ is tall if and only if $\lim_{n\to\infty}a_n=0$.

An ideal $\I$ is a \emph{P-ideal} if for every sequence $(A_n)_{n\in\N}$ of elements of $\I$ there is $A\in\I$ such that $A_n\setminus A$ is finite for all $n\in\N$. It is easy to verify that all summable ideals are P-ideals.

\section{Submeasures and objects induced by them}

\subsection{Submeasures}

A function $\phi:\mathcal{P}(\N)\to[0,\infty]$ is called a \emph{submeasure} if $\phi(\emptyset)=0$, $\phi(\{n\})<\infty$ for every $n\in\N$, and 
\[
\phi(C)\leq\phi(C\cup D)\leq\phi(C)+\phi(D)
\]
for all $C,D\subseteq\N$. A submeasure $\phi$ is \emph{lower semicontinuous} (lsc, in short) if $\phi(C)=\lim_{n\to\infty}\phi(C\cap\{1,2,\ldots,n\})$ for each $C\subseteq\N$.

An lsc submeasure $\phi$ is \emph{non-pathological}, if
$$\phi(C)=\sup\{\mu(C):\ \mu\text{ is a measure such that }\mu\leq\phi\},$$
for all $C\subseteq\N$. Not every lsc submeasure is non-pathological -- see \cite[Section~1.9]{farah-book}, \cite[Theorem~4.12]{ft-mazur}, \cite[Section~6.2]{marciszewski-sobota}, \cite{MezaPat} or \cite[Theorem~4.7]{tryba-gendensity} for such examples. 

Given an lsc submeasure $\phi$, let $\mathcal{M}_\phi$ be the family of all measures $\mu$ on $\N$ such that $\mu\leq\phi$. Then, by definition, $\phi$ is non-pathological if and only if $\phi(C)=\sup\{\mu(C): \mu\in\mathcal{M}_\phi\}$ for all $C\subseteq\N$.

\subsection{Ideals induced by submeasures}

By identifying subsets of $\N$ with their characteristic functions, we can treat ideals as subsets of the Cantor space $\{0,1\}^\N$. Mazur in \cite[Lemma 1.2]{Mazur} proved that an ideal is $\bf{F_\sigma}$ if and only if it is of the form:
$$\Fin(\phi)=\left\{C\subseteq\N:\ \phi(C)<\infty\right\}$$
for some lower semicontinuous submeasure $\phi$ such that $\N\notin\Fin(\phi)$ (see also \cite[Theorem 1.2.5]{Farah}).

Solecki in \cite[Theorem 3.1]{SoleckiExh} showed that an ideal is an analytic P-ideal if and only if it is of the form:
$$\Exh(\phi)=\left\{C\subseteq\N:\ \lim_{n\to\infty}\phi(C\setminus \{1,2,\ldots,n\})=0\right\}$$
for some lower semicontinuous submeasure $\phi$ such that $\N\notin\Exh(\phi)$ (see also \cite[Theorem 1.2.5]{Farah}). 

It is easy to see that $\Exh(\phi)\subseteq\Fin(\phi)$ for every lsc submeasure $\phi$. Moreover, for every lsc submeasure $\phi$ we can find an lsc submeasure $\phi'$ such that $\Fin(\phi)=\Fin(\phi')$, $\Exh(\phi)=\Exh(\phi')$ and additionally $\phi'(\{k\})>0$ for all $k\in\N$ (it suffices, for instance, to put $\phi'(C)=\phi(C)+\sum_{n\in C}\frac{1}{2^n}$ for all $C\subseteq\N$).

Note that every summable ideal is of the form $\Exh(\phi_A)$ as well as of the form $\Fin(\phi_A)$, where $\phi_A(C)=\sum_{n\in C}a_n$. For more examples of ideals induced by submeasures see \cite[Example 1.2.3]{Farah}. 

\subsection{Banach spaces of real sequences}

\begin{definition}
Let $\phi$ be a non-pathological lsc submeasure. Define a function $\hat\phi:\R^\N\to[0,\infty]$ by:
\[
\hat\phi(x)=\sup\left\{\sum_{n\in\N}\mu(\{n\})|x_n|: \mu\in\mathcal{M}_\phi\right\}
\]
for all $x=(x_n)_{n\in\N}\in\R^\N$. Define also:
\[
\FIN(\phi)=\left\{x\in\R^\N: \hat\phi(x)<\infty\right\};
\]
\[
\EXH(\phi)=\left\{x\in\R^\N: \lim_{n\to\infty}\hat\phi(x\cdot\chi_{\{n,n+1,\ldots\}})=0\right\}.
\]
Moreover, let $\mFIN(\phi)=\FIN(\phi)\cap\Mon$ and $\mEXH(\phi)=\EXH(\phi)\cap\Mon$, where
\[
\Mon=\{(x_n)_{n\in\N}\in\R^\N: |x_{n+1}|\leq|x_n|\text{ for all }n\in\N\}.
\]
\end{definition}

Spaces $\FIN(\phi)$ and $\EXH(\phi)$ were introduced in \cite[Section 5]{FBN}. Note that $\hat\phi(x)=\lim_{n\to\infty}\hat\phi(x\cdot\chi_{\{1,2,\ldots,n\}})$ for every $x\in\R^\N$ (see \cite[Proposition 5.3]{FBN}).

\begin{example}[{\cite[Examples 5.6 and 5.7]{FBN}}]\
\begin{itemize}
    \item Consider the submeasure given by:
\[
\phi(C)=\begin{cases}
    1, & \text{if }C\neq\emptyset,\\
    0, & \text{if }C=\emptyset,
\end{cases}
\]
for every $C\subseteq\N$. Then $\FIN(\phi)=\ell_\infty$ and $\EXH(\phi)=c_0$.
\item Consider the submeasure given by $\phi(C)=|C|$ for every $C\subseteq\N$. Then $\FIN(\phi)=\EXH(\phi)=\ell_1$.
\end{itemize}
\end{example} 

\begin{remark}
Notice that for every $x=(x_n)_{n\in\N}\in\R^\N$ and $k\in\N$ we have $\hat\phi(x\cdot \chi_{\{k\}})=\phi(\{k\})\cdot |x_k|$. Therefore, for any $A\subseteq\N$ we have
$$\sup_{k\in A}\left( \phi(\{k\})\cdot |x_k|\right)  \leq \hat\phi(x\cdot \chi_A)\leq  \sup_{k\in A} |x_k| \cdot \sum_{k\in A} \phi(\{k\})  $$
\end{remark}

\begin{remark}
\label{rem1}
Observe that $C\in\Fin(\phi)$ if and only if $\chi_C\in\FIN(\phi)$, where $\chi_C$ denotes the characteristic function of $C$. Similar equivalence holds for $\Exh(\phi)$ and $\EXH(\phi)$.
\end{remark}

\begin{proposition}{\cite[Propositions 5.1 and 5.3]{FBN}}
\label{FBN}
Suppose that $\phi$ is a non-pathological lsc submeasure. Then $\FIN(\phi)$ and $\EXH(\phi)$ are Banach spaces normed by $\hat\phi$. Moreover, $\EXH(\phi)\subseteq\FIN(\phi)$.
\end{proposition}

\begin{proposition}
\label{prop:mEXHsubsetmFIN}
Let $\phi$ be a non-pathological lsc submeasure such that $\phi(\{k\})>0$ for all $k\in\N$. Then $\mFIN(\phi)$ is a closed subspace of $\FIN(\phi)$ and $\mEXH(\phi)$ is a closed subspace of $\EXH(\phi)$. In particular, $\mFIN(\phi)$ and $\mEXH(\phi)$ are Banach spaces normed by $\hat\phi$. Moreover, $\mEXH(\phi)\subseteq\mFIN(\phi)$.
\end{proposition}

\begin{proof}
The inclusion $\mEXH(\phi)\subseteq\mFIN(\phi)$ follows from Proposition \ref{FBN}. 

Actually, it suffices to show that $\Mon$ is closed in $\FIN(\phi)$ and in $\EXH(\phi)$. Let $x=(x_n)\in \R^\N\setminus \Mon$. Then there is $n\in\N$ such that $|x_{n+1}|>|x_n|$. Define $r=\frac{|x_{n+1}|-|x_n|}{3}\min\{\hat\phi(e_n),\hat\phi(e_{n+1})\}$, where $e_n\in\R^\N$ is the sequence given by:
\[
(e_n)_i=\begin{cases}
    1, & \text{if }i=n,\\
    0, & \text{otherwise.}
\end{cases}
\]
Note that $r>0$ since $\hat\phi(e_k)=\phi(\{k\})>0$ for all $k\in\N$.

We claim that if $y\in\R^\N$ is such that $\hat\phi(x-y)<r$, then $r\notin \Mon$ (which shows that $\Mon$ is closed in $\FIN(\phi)$ and in $\EXH(\phi)$). Indeed, if $\hat\phi(x-y)<r$, then:
\[
|x_n-y_n|\hat\phi(e_n)\leq\hat\phi(x-y)<r\leq\frac{|x_{n+1}|-|x_n|}{3}\hat\phi(e_n),
\]
so $|x_n-y_n|<\frac{|x_{n+1}|-|x_n|}{3}$. Similarly, $|x_{n+1}-y_{n+1}|<\frac{|x_{n+1}|-|x_n|}{3}$. Hence, $|y_{n+1}|-|y_n|>\frac{|x_{n+1}|-|x_n|}{3}$ and we get that $|y_{n+1}|>|y_n|$, which shows $y\notin \Mon$.
\end{proof}

\begin{proposition}{\cite[Theorem 5.4]{FBN}}
\label{Exh=Fin}
The following are equivalent for any non-pathological lsc submeasure $\phi$:
\begin{itemize}
    \item $\Exh(\phi)=\Fin(\phi)$;
    \item $\EXH(\phi)=\FIN(\phi)$;
    \item $\FIN(\phi)$ is separable.
\end{itemize}
\end{proposition}

\begin{remark}
\label{Exh=Fin-rem}
Obviously, if $\EXH(\phi)=\FIN(\phi)$, then also $\mEXH(\phi)=\mFIN(\phi)$.
\end{remark}

\subsection{Variations}

\begin{definition}
Let $\phi$ be a non-pathological lsc submeasure. For $J=(J_n)\in\mathcal{P}_I$ denote by $x(J)$ the sequence $(|x(J_n)|)$. Define:
\[
\BV(\phi)=\left\{x\in B(I): \sup_{J\in\mathcal{P}_I}\hat\phi(x(J))<\infty\right\}.
\]
\end{definition}

\begin{remark}The requirement $x\in B(I)$ may be removed. Indeed, assume that 
there exists such a sequence $(t_n)\subseteq I$ that $|x(t_n)|\to+\infty$. Fix any $k\in\N$ such that $\phi(\{k\})>0$. Then if we take for each $n\in\N$ any $J^n=(J^n_i)\in\mathcal{P}_I$ such that $J^n_k = [0,t_n]$ we have
\[ \sup_{J\in\mathcal{P}_I}\hat\phi(x(J)) \geq |x(t_n)-x(0)|\phi(\{k\}) \to +\infty,  \]
which means that the condition $\sup_{J\in\mathcal{P}_I}\hat\phi(x(J))<\infty$ will not be satisfied anyway.
\end{remark}

\begin{proposition}
Let $\phi$ be a non-pathological lsc submeasure. Then $\BV(\phi)$ is a Banach space normed by 
\[
\|x\|_\phi=|x(0)|+\sup_{J\in\mathcal{P}_I}\hat\phi(x(J)),
\]
for all $x\in\BV(\phi)$. 
\end{proposition}

\begin{proof} 
Let us assume that $x\in B(I)$ is such that $x\neq 0$ and $x(0)=0$. Then,  there exists such $t\in I$ that $x(t)\neq 0$. Let us assume that $\phi(\{k\})>0$ for certain $k\in\N$ and let $J\in\mathcal{P}_I$ be such sequence of intervals that $J_k=[0,t]$. Then
\[
\hat\phi(|x(J)|) \geq |x(t)-x(0)|\phi(\{k\}) > 0.
\]

Let $c\in\R$ be any constant. The condition $\hat\phi(|c\cdot x(J)|) = c\hat\phi(|x(J)|)$ is obvious for any $J\in\mathcal{P}_I$. Similarly the triangle inequality
\[ \hat\phi(|(x+y)(J)|) \leq \hat\phi(|x(J)|) + \hat\phi(|y(J)|)
\]
for any functions $x,y\colon I\to\R$.
Passing to the supremum for $J\in {\mathcal P}_I$ keeps these conditions. 

Now it remains to prove that the space is complete. The proof will be a standard one. Let us take a Cauchy sequence $(x_n)\subseteq \BV(\phi)$. Let us fix $\varepsilon>0$ and take such $n,m\geq N$ that 
\[
|x_n(0)-x_m(0)| + \sup_{J\in\mathcal{P}_I}\hat\phi((x_n-x_m)(J)) \leq \varepsilon.
\]
First of all, let us observe that it means that the sequence $(x_n(0))$ is a real Cauchy sequence so it converges to some real number $x(0)$. We can also observe that the sequence $x_n(t)$ converges to $x(t)$ for any $t\in I$. Indeed, fix any $t\in I$ and as before, let us assume that $\phi(\{k\})>0$ for certain $k\in\N$. Let $J\in\mathcal{P}_I$ be such sequence of intervals that $J_k=[0,t]$. Then
\[
|(x_n-x_m)(J_k)|\phi(\{k\}) \leq \varepsilon
\]
and
\[
|x_n(t)-x_m(t)| \leq\frac{1}{\phi(\{k\})} \varepsilon + |x_n(0) - x_m(0)|,
\]
which eventually proves that $(x_n(t))$ is a Cauchy sequence, so it converges to some $x(t)$. 

Now, as we have the pointwise limit $x(t)$, we are going to show that $x\in BV(\phi)$ and that 
$\|x_n-x\|_\phi\to 0$ as $n\to +\infty$. Let us take any $\mu\in \mathcal{M}_\phi$, any natural number $K\in \N$ and any sequence of intervals $J\in{\mathcal P}_I$. Then we have
\[
\sum_{k=1}^K |(x_n-x_m)(J_k)|\mu(\{k\}) \leq \varepsilon
\]
and we may pass to the limit with $m\to+\infty$ (as we know the sequence $x_m(t))$ converges pointwise to $x(t)$). This gives us
\[
\sum_{k=1}^K |(x_n-x)(J_k)|\mu(\{k\}) \leq \varepsilon.
\]
Since the inequality holds for all $K$, $\mu$ and $J$ we can see that 
\[
\sup_{J\in\mathcal{P}_I}\hat\phi((x_n-x)(J)) \leq \varepsilon.
\]
To show that $x\in \BV(\phi)$ it is enough to take any fixed $x_n\in\BV(\phi)$ such that 
\[
\sup_{J\in\mathcal{P}_I}\hat\phi((x_n-x)(J)) \leq 1
\]
and refer to seminorm properties to see that 
\[
\sup_{J\in\mathcal{P}_I}\hat\phi(x(J)) - \sup_{J\in\mathcal{P}_I}\hat\phi(x_n(J)) \leq \sup_{J\in\mathcal{P}_I}\hat\phi((x_n-x)(J)) \leq 1
\]
\end{proof}

\begin{proposition}
Let $\phi$ be a non-pathological lsc submeasure.
\begin{itemize}
    \item[(a)] If the sequence $\phi(\{k\})$ is unbounded, then $BV(\phi)$ reduces to the space of constant functions.
    \item[(b)] If the sequence $\phi(\{k\})$ is bounded, then the space $\BV(\phi)$ contains the space of functions of bounded classical Jordan variation as a subset.
\end{itemize}
\end{proposition}

\begin{proof}
(a): If $x\in B(I)$ is such that $|x(s)-x(t)|=a>0$ for some $[s,t]\in (0,1)$, then taking for each $k\in\N$ such $J^k=(J^k_n)_{n\in\N}\in{\mathcal P}_I$ that $J^k_k = [s,t]$ we get 
\[
\hat\phi(x(J)) \geq a\phi(\{k\}).
\]
Thus the variation $\sup_{J\in\mathcal{P}_I}\hat\phi(|x(J)|)$ is unbounded.

(b): Let us take any $x\in \BV(I)$ and any $J\in{\mathcal P}_I$. Then 
\[
\hat\phi(|x(J)|) \leq \sum_{k\in\N} \phi(\{k\})|x(J_k)|) \leq M \var(x).
\]
The last inequality actually shows that the space $\BV(I)$ is continuously embedded in $\BV(\phi)$.
\end{proof}

\begin{remark} Let us observe that if $\phi(\N)<+\infty$, then $BV(\phi)=B(I)$. To see this let us take any bounded function $x\in B(I)$ such that $|x(t)|\leq M$ for all $t\in I$. Then for any interval $J\subset I$ we get $|x(J)|\leq 2M$ and for any $\mu\in{\mathcal M}_\phi$ and any  $(J_k)\in{\mathcal P}_I$ we have
\[
\sum_{k\in N} |x(J_k)|\mu\{k\} \leq 2M \mu(\N)
\]
leading to $\hat\phi(x) \leq 2M \phi(\N)<+\infty.$
\end{remark}

\begin{proposition}Let $\phi$ be a non-pathological lsc submeasure such that $\phi(\N)=+\infty$. Then the space $BV(\phi)$ is a subset of the space of all {\em bounded regulated functions} (i.e. bounded functions having finite left and right limit in every point of their domain).
\end{proposition}
\begin{proof}Assume, contrary to our claim, that there exist monotone sequences $(t_n)\subseteq I$ and $(s_n)\subseteq I$ converging to $a\in I$ from the same side and such that $x(s_n) - x(t_m)\geq \delta > 0$ for all $n,m\in\N$. Then, taking subsequences if necessary, we have the sequence of nonoverlapping intervals $I_n=[s_n,t_n]\subset I$ such that $|x(I_n)|\geq \delta$. Then for any $\mu\in {\mathcal M}_\phi$ we have
\[
\sum_{n\in\N} |x(I_n)|\mu(\{n\}) \geq \delta \mu(\N)
\]
so
\[
\hat\phi(|x(I_n)|) \geq \delta \phi(\N) = +\infty.
\]    
\end{proof}

\begin{remark}\label{def_var_permutation}
For any non-pathological lsc submeasure $\phi$ and any permutation $\pi\colon\N\to\N$, the function $\psi:\mathcal{P}(\N)\to[0,\infty]$ given by
\[
\psi(C) = \phi(\pi[C]), 
\]
for all $C\subseteq\N$, also is a non-pathological lsc submeasure. What is more, we can easily see that  $\BV(\phi) = \BV(\psi)$.
\end{remark}

\begin{proposition}\label{prop:phi=psi}
Let $\phi$ and $\psi$ be two non-pathological submeasures. If there exists $M>0$ such that for every $A\subseteq\N$ we have $|\phi(A)-\psi(A)|\leq M$ then $\BV(\phi)=\BV(\psi)$.
\end{proposition}
\begin{proof}
Suppose to the contrary that there exists $x\in\BV(\psi)\setminus \BV(\phi)$. Then $x$ is bounded by some $N> 0$ and there exists $M_1>0$ such that for every $J\in \cP_I$ we have $\hat\psi(x(J))\leq M_1$.

On the other hand, there exists a measure $\mu\leq \phi$, $J\in \cP_I$ and $k\in\N$ such that $\sum_{n\leq k} \mu(\{n\}) |x(J_n)|> 2NM+2N+M_1$.
Meanwhile, there exists a measure $\nu\leq\psi$ such that $\sum_{n\leq k}\nu(\{n\})=\psi(\{1,\ldots,k\})$. Clearly,  we have $\sum_{n\leq k} \nu(\{n\}) |x(J_n)|\leq M_1$. Therefore, we obtain
$$\sum_{n\leq k} (\mu(\{n\})-\nu(\{n\})) |x(J_n)|> 2NM+2N +M_1-M_1= 2N(M+1). $$
Since $|x(J_n)|\leq 2N$ for every $n$, it follows that 
$$\sum_{n\leq k} (\mu(\{n\})-\nu(\{n\})) \geq  \frac{\sum_{n\leq k} (\mu(\{n\})-\nu(\{n\})) |x(J_n)|}{2N}\geq M+1. $$
Therefore, 
$$\phi(\{1,\ldots,k\}) -\psi(\{1,\ldots,k\}) \geq \sum_{n\leq k} (\mu(\{n\})-\nu(\{n\}))\geq M+1>M,$$
a contradiction.
\end{proof}

\begin{corollary}
For every non-pathological submeasure $\phi$ there exists a non-pathological submeasure $\psi$ such that $\BV(\phi)=\BV(\psi)$ and the sequence $(\psi(\{n\}))$ does not tend to zero.
\end{corollary}
\begin{proof}
Define the function $\delta:\cP(\N)\to[0,\infty]$ by
$$\delta(A)=\begin{cases}
    0, & \text{if }A=\emptyset,\\
    1, & \text{if }A\not=\emptyset.
\end{cases}
$$
Clearly, $\delta$ is a non-pathological submeasure. Define the function $\psi:\cP(\N)\to[0,\infty]$ by $\psi(A)=\max\{\phi(A),\delta(A)\}$ for every $A\subseteq\N$. Then $\psi$ is  non-pathological submeasure as a maximum of two non-pathological submeasures. It is also clear that $(\psi(\{n\}))$ does not tend to zero because $\psi(\{n\})\geq \delta(\{n\})=1$ for every $n\in\N$. Moreover, since $|\phi(A)-\psi(A)|\leq 1$ for every $A\subseteq\N$, by Proposition~\ref{prop:phi=psi} we obtain $\BV(\phi)=\BV(\psi)$.
\end{proof}

\subsection{Basic results}

\begin{proposition}
\label{prop-basic-1}
For any non-pathological lsc submeasure $\phi$ the following are equivalent :
\begin{itemize}
    \item[(a)] $\phi(\N)<\infty$;
    \item[(b)] $\Fin(\phi)=\mathcal{P}(\N)$;
    \item[(c)] $\FIN(\phi)\supseteq\ell_\infty$;
    \item[(d)] $\mFIN(\phi)=\Mon$;
    \item[(e)] $\mFIN(\phi)\not\subseteq c_0$;
    \item[(f)] every bounded function $x:[0,1]\to\mathbb{R}$ belongs to $\BV(\phi)$.
\end{itemize}
\end{proposition}

\begin{proof}
(a)$\implies$(c): If $x\in\ell_\infty$, then there is $M>0$ such that $|x_n|\leq M$ for all $n\in\N$. Hence, if $\hat{1}\in\mathbb{R}^\N$ is the infinite sequence constantly equal to $1$, then
\[
\hat{\phi}(x)\leq\hat{\phi}(M\cdot \hat{1})=M\hat{\phi}(\hat{1})=M\phi(\N)<\infty.\]
Thus $x\in\FIN(\phi)$. 

(c)$\implies$(d): If $\FIN(\phi)\supseteq\ell_\infty$, then 
\[
\Mon\supseteq\mFIN(\phi)=\FIN(\phi)\cap \Mon\supseteq\ell_\infty\cap \Mon=\Mon.
\]

(d)$\implies$(e): The sequence constantly equal to one belongs to $\Mon=\mFIN(\phi)$, but not to $c_0$.

(e)$\implies$(b): Since $\mFIN(\phi)\not\subseteq c_0$, there is some $x\in\mFIN(\phi)\setminus c_0$. Since $\mFIN(\phi)\subseteq\Mon$, $|x|$ converges to some $l>0$ and $|x_n|\geq l$ for all $n\in\N$. Thus, if $\hat{1}$ denotes the infinite sequence constantly equal to $1$, then $\phi(\N)=\hat{\phi}(\hat{1})=\frac{1}{l}\hat{\phi}(l\cdot\hat{1})\leq\frac{1}{l}\hat{\phi}(x)<\infty$, so $\N\in\Fin(\phi)$ and consequently $\Fin(\phi)=\mathcal{P}(\N)$.

(b)$\implies$(f): If $x:[0,1]\to\mathbb{R}$ is bounded by some $M>0$, then $|x(J)|$ is bounded by $2M$ for every interval $J\subseteq I$. Hence, since $\N\in\Fin(\phi)$, we have $\sup_{J\in\mathcal{P}_I}\hat{\phi}(|x(J)|)\leq 2 M\phi(\N)<\infty$.

(f)$\implies$(a): Consider the Dirichlet function $x_D:[0,1]\to\mathbb{R}$ given by:
\[
x_D(t)=\begin{cases}
1, & \text{if }t\in\mathbb{Q}\cap[0,1];\\
0, & \text{otherwise.}
\end{cases}
\]
Then $x_D$ is bounded, so it belongs to $\BV(\phi)$. Note also that there is $J\in\mathcal{P}_I$ such that $x_D(J)$ is constantly equal to $1$. Thus, $\phi(\N)=\hat{\phi}(x_D(J))<\infty$.
\end{proof}

\begin{proposition}\label{prop: mexh-basic}
The following are equivalent for any non-pathological lsc submeasure $\phi$:
\begin{itemize}
    \item[(a)] $\lim_n\phi(\{n,n+1,\ldots\})=0$;
    \item[(b)] $\Exh(\phi)=\mathcal{P}(\N)$;
    \item[(c)] $\EXH(\phi)\supseteq\ell_\infty$;
    \item[(d)] $\mEXH(\phi)=\Mon$;
    \item[(e)] $\mEXH(\phi)\not\subseteq c_0$.
\end{itemize}
\end{proposition}

\begin{proof}
(a)$\implies$(c): If $x\in\ell_\infty$, then there is $M>0$ such that $|x_n|\leq M$ for all $n\in\N$. Hence, if $\hat{1}\in\mathbb{R}^\N$ is the infinite sequence constantly equal to $1$, then
\[
\hat{\phi}(x\cdot\chi_{\{n,n+1,\ldots\}})\leq\hat{\phi}(M\cdot \hat{1}\cdot\chi_{\{n,n+1,\ldots\}})=M\phi(\{n,n+1,\ldots\})\to 0.
\]
Thus $x\in\EXH(\phi)$. 

(c)$\implies$(d): If $\EXH(\phi)\supseteq\ell_\infty$, then 
\[
\Mon\supseteq\mEXH(\phi)=\EXH(\phi)\cap \Mon\supseteq\ell_\infty\cap \Mon=\Mon.
\]

(d)$\implies$(e): The sequence constantly equal to one is in $\Mon=\mEXH(\phi)$, but not in $c_0$.

(e)$\implies$(b): Since $\mEXH(\phi)\not\subseteq c_0$, there is some $x\in\mEXH(\phi)\setminus c_0$. Since $\mEXH(\phi)\subseteq\Mon$, $x$ converges to some $l>0$ and $x_n\geq l$ for all $n\in\N$. Thus, if $\hat{1}$ denotes the infinite sequence constantly equal to $1$, then $\phi(\N\setminus\{1,2,\ldots,n-1\})=\phi(\{n,n+1,\ldots\})=\hat{\phi}(\hat{1}\cdot\chi_{\{n,n+1,\ldots\}})=\frac{1}{l}\hat{\phi}(l\cdot\hat{1}\cdot\chi_{\{n,n+1,\ldots\}})\leq\frac{1}{l}\hat{\phi}(x\cdot\chi_{\{n,n+1,\ldots\}})\to 0$, so $\N\in\Exh(\phi)$ and consequently $\Exh(\phi)=\mathcal{P}(\N)$.

(b)$\implies$(a): Since $\Exh(\phi)=\mathcal{P}(\N)$, $\N\in\Exh(\phi)$, which means that $\lim_n\phi(\{n,n+1,\ldots\})=0$.
\end{proof}

\section{Inclusions in the general case}

\subsection{Two orders}

Actually, $\hat\phi$ is a function defined on infinite sequences of reals. However, for simplicity, we will sometimes write $\hat\phi(x)$ even for finite sequences $x$ -- in such cases we mean $\hat\phi(x^\frown 0)$, where $x^\frown 0$ is the infnite sequence starting with $x$ and followed by zeros.

\begin{definition}
Let $\phi_1$ and $\phi_2$ be two non-pathological lsc submeasures. We write:
\begin{itemize}
    \item $\phi_2\preceq\phi_1$ if there is $M>0$ such that $\hat\phi_2(x)\leq M\hat\phi_1(x)$ for every finite sequence $x\in\bigcup_{n\in\N}\R^n$;
    \item $\phi_2\preceq_m\phi_1$ if there is $M>0$ such that $\hat\phi_2(x)\leq M\hat\phi_1(x)$ for every non-increasing finite sequence $x\in\bigcup_{n\in\N}\R^n$.
\end{itemize}
\end{definition}

\subsection{Ideals}

\begin{proposition}
\label{general-ideals}
Let $\phi_1$ and $\phi_2$ be two non-pathological lsc submeasures.
\begin{itemize}
    \item[(a)] If $\phi_2\preceq\phi_1$ then $\Exh(\phi_1)\subseteq\Exh(\phi_2)$ and $\Fin(\phi_1)\subseteq\Fin(\phi_2)$
    \item[(b)] If either $\Exh(\phi_1)\subseteq\Exh(\phi_2)$ or $\Fin(\phi_1)\subseteq\Fin(\phi_2)$, then $\Exh(\phi_1)\subseteq\Fin(\phi_2)$.
    \item[(c)] $\Exh(\phi_1)\subseteq\Fin(\phi_2)$ if and only if
    \[
    \exists_{M>0}\ \forall_{F\in\Fin}\ \phi_1(F)>\frac{1}{M}\text{ or }\phi_2(F)<M.
    \]
\end{itemize}
\end{proposition}

\begin{proof}
(a): This is clear with the use of Remark \ref{rem1}.

(b): This follows from the inclusions $\Exh(\phi_1)\subseteq\Fin(\phi_1)$ and $\Exh(\phi_2)\subseteq\Fin(\phi_2)$ (see Proposition \ref{FBN}).

(c): Suppose first that there is $M>0$ such that for all $F\in\Fin$ either $\phi_1(F)>\frac{1}{M}$ or $\phi_2(F)<M$. Let $C\in\Exh(\phi_1)$. Then there is $k\in\N$ such that $\phi_1(C\setminus\{1,2,\ldots,k\})<\frac{1}{M}$. Hence, if $F\subseteq C\setminus\{1,2,\ldots,k\}$ is finite, then $\phi_1(F)<\frac{1}{M}$ and $\phi_2(F)<M$ (by our assumption). Thus, since $\phi_2$ is lsc, we have:
\begin{align*}
\phi_2(C) & \leq\phi_2(C\cap\{1,2,\ldots,k\})+\phi_2(C\setminus\{1,2,\ldots,k\})\cr 
& \leq \left(\sum_{i\leq k}\phi_2(\{i\})\right)+M<\infty.
\end{align*}

Suppose now that for each $n\in\N$ there is $F_n\in\Fin$ such that $\phi_1(F_n)\leq\frac{1}{2^n}$ and $\phi_2(F_n)\geq 2^n$. Then clearly $C=\bigcup_{n\in\N}F_n\notin\Fin(\phi_2)$. On the other hand, we will show that $C\in\Exh(\phi_1)$. Given any $\varepsilon>0$ there is $n_0\in\N$ such that $\frac{1}{2^{n_0}}<\varepsilon$. Find $k_0\in\N$ such that $F_1\cup\ldots\cup F_{n_0}\subseteq[1,k_0]$ and observe that for each $k>k_0$ we have:
\begin{align*}
\phi_1(C\setminus\{1,2,\ldots,k\}) & \leq\phi_1(C\setminus\{1,2,\ldots,k_0\})\cr 
& \leq \sum_{i=n_0+1}^\infty \phi_1(F_i)\leq\sum_{i=n_0+1}^\infty \frac{1}{2^i}=\frac{1}{2^{n_0}}<\varepsilon.
\end{align*}
\end{proof}

\begin{remark}
Observe that there are non-pathological lsc submeasures $\phi_1$ and $\phi_2$ such that $\Exh(\phi_1)\subseteq\Exh(\phi_2)$, but $\Fin(\phi_1)\not\subseteq\Fin(\phi_2)$. Indeed, this is true for 
\[
\phi_1(C)=\begin{cases}
    1, & \text{if }C\neq\emptyset,\\
    0, & \text{if }C=\emptyset,
\end{cases}
\]
and $\phi_2(C)=\sum_{i\in C}\frac{1}{i}$, since in this case $\Exh(\phi_1)=\Fin\subseteq\I_{1/n}=\Exh(\phi_2)$ and $\Fin(\phi_1)=\mathcal{P}(\N)\not\subseteq\I_{1/n}=\Fin(\phi_2)$. 

On the other hand, there are also non-pathological lsc submeasures $\psi_1$ and $\psi_2$ such that $\Exh(\psi_1)\not\subseteq\Exh(\psi_2)$, but $\Fin(\psi_1)\subseteq\Fin(\psi_2)$. This is the case for $\psi_1=\phi_2$ and $\psi_2=\phi_1$ (where $\phi_1$ and $\phi_2$ are as in the previous paragraph).
\end{remark}

\subsection{Spaces of real sequences}

\begin{lemma}
\label{lem-L}
Let $\phi_1$ and $\phi_2$ be two non-pathological lsc submeasures and assume that $\phi_1(\{i\})>0$ and $\phi_2(\{i\})>0$ for every $i\in\N$. Then for every $m\in\N$ there is $L>0$ such that $\hat\phi_2(y)\leq L\hat\phi_1(y)$ for every $y\in\R^{m}$.
\end{lemma}

\begin{proof}
Define:
\[
L=\frac{\sum_{i\leq m}\phi_2(\{i\})}{\min_{i\leq m}\phi_1(\{i\})}.
\]
Then $L>0$. Fix any $y\in\R^m$. If $y$ is constantly equal to zero, then $\hat\phi_2(y)=0\leq L\hat\phi_1(y)$ and we are done. Otherwise, find $r=\max_{i\leq m}|y_i|$ and let $j\leq m$ be such that $r=|y_j|$. Note that:
\begin{align*}
\hat\phi_2\left(\frac{y}{r}\right) & \leq\sum_{i\leq m}\phi_2(\{i\})\frac{|y_i|}{r}\leq\sum_{i\leq m}\phi_2(\{i\})= L\min_{i\leq m}\phi_1(\{i\})\cr 
& \leq L\frac{|y_j|}{r}\phi_1(\{j\})\leq L\hat\phi_1\left(\frac{y}{r}\right).    
\end{align*}
Hence, after multiplying by $r$ we get that $\hat\phi_2(y)\leq L\hat\phi_1(y)$.
\end{proof}

\begin{theorem}
\label{general-sequences}
Let $\phi_1$ and $\phi_2$ be two non-pathological lsc submeasures and assume that $\phi_1(\{i\})>0$ and $\phi_2(\{i\})>0$ for every $i\in\N$. The following are equivalent:
\begin{itemize}
    \item[(a)] $\phi_2\preceq\phi_1$;
    \item[(b)] $\EXH(\phi_1)\subseteq\EXH(\phi_2)$;
    \item[(c)] $\FIN(\phi_1)\subseteq\FIN(\phi_2)$;
    \item[(d)] $\EXH(\phi_1)\subseteq\FIN(\phi_2)$;
\end{itemize}
\end{theorem}

\begin{proof}
(a)$\implies$(b) and (a)$\implies$(c) are clear, since $\hat\phi(x)=\lim_{n\to\infty}\hat\phi(x\cdot\chi_{\{1,2,\ldots,n\}})$ for every $x\in\R^\N$.

(b)$\implies$(d) and (c)$\implies$(d) follow from $\EXH(\phi_1)\subseteq\FIN(\phi_1)$ and $\EXH(\phi_2)\subseteq\FIN(\phi_2)$ (see Proposition \ref{FBN}).

(d)$\implies$(a): Suppose that $\phi_2\not\preceq\phi_1$, i.e., for every $n\in\N$ there is a finite real sequence $z$ such that $\hat\phi_2(z)>2^{2n+1}\hat\phi_1(z)$.

We will recursively construct sequences $(n_k),(m_k),(L_k)\subseteq\N$ and $(x_k)\subseteq\bigcup_{n\in\N}\R^n$ such that for each $k\in\N$:
\begin{itemize}
    \item[(i)] $n_1=1$ and $n_{k+1}>n_k$;
    \item[(ii)] $2^{2n_{k+1}}>L_k$;
    \item[(iii)] $\hat\phi_2(y)\leq L_k\hat\phi_1(y)$ for every $y\in\R^{m_k}$;
    \item[(iv)] $m_k$ is the length of $x_k$;
    \item[(v)] $\hat\phi_1(x_k\chi_{\{m_{k-1}+1,m_{k-1}+2,\ldots,m_k\}})=\frac{1}{2^{n_k}}$;
    \item[(vi)] $\hat\phi_2(x_k\chi_{\{m_{k-1}+1,m_{k-1}+2,\ldots,m_k\}})>2^{n_k}$.
\end{itemize}

Start with $n_1=1$ (so that item (i) is met; item (ii) is empty in this case), using our assumption we find some finite real sequence $z_1$ such that $\hat\phi_2(z_1)>2^{2n_1+1}\hat\phi_1(z_1)$ and put $x_1=\frac{z_1}{2^{n_1}\hat\phi_1(z_1)}$. Let $m_1$ be the length of $x_1$. Note that items (v) and (vi) are satisfied, since:
\[
\hat\phi_2(x_1)=\frac{\hat\phi_2(z_1)}{2^{n_1}\hat\phi_1(z_1)}>\frac{2^{2n_1+1}\hat\phi_1(z_1)}{2^{n_1}\hat\phi_1(z_1)}=2^{n_1+1}>2^{n_1}.
\]
Using Lemma \ref{lem-L}, we can find $L_1\in\N$ satisfying item (iii).

If $n_i,m_i,L_i$ and $x_i$ for all $i\leq k$ are already defined, we can find $n_{k+1}$ such that items (i) and (ii) are met. By our assumption, there is some finite real sequence $z_{k+1}$ such that $\hat\phi_2(z_{k+1})>2^{2n_{k+1}+1}\hat\phi_1(z_{k+1})$. Let $m_{k+1}$ be the length of $z_{k+1}$. Observe that $m_{k+1}>m_k$ (by items (ii) and (iii)). Moreover, $\hat\phi_2(z_{k+1}\cdot\chi_{\N\setminus\{1,2,\ldots,m_k\}})>\left(2^{2n_{k+1}+1}-L_k\right)\hat\phi_1(z_{k+1}\cdot\chi_{\N\setminus\{1,2,\ldots,m_k\}})>2^{2n_{k+1}}\hat\phi_1(z_{k+1}\cdot\chi_{\N\setminus\{1,2,\ldots,m_k\}})$. Indeed, the second inequality follows from item (ii) and if the first one would be false, using item (iii) we should get:
\begin{align*}
    \hat\phi_2(z_{k+1}) & \leq \hat\phi_2(z_{k+1}\cdot\chi_{\{1,2,\ldots,m_k\}})+\hat\phi_2(z_{k+1}\cdot\chi_{\N\setminus\{1,2,\ldots,m_k\}})\\ 
    &\leq L_k\hat\phi_1(z_{k+1}\cdot\chi_{\{1,2,\ldots,m_k\}})+\left(2^{2n_{k+1}+1}-L_k\right)\hat\phi_1(z_{k+1}\cdot\chi_{\N\setminus\{1,2,\ldots,m_k\}})\\
    & \leq \left(L_k+\left(2^{2n_{k+1}+1}-L_k\right)\right)\hat\phi_1(z_{k+1})=2^{2n_{k+1}+1}\hat\phi_1(z_{k+1}),
\end{align*}
which contradicts the choice of $z_{k+1}$. Put:
\[
(x_{k+1})_i=\begin{cases}
    (x_k)_i, & \text{if }i\leq m_k,\\
    \frac{(z_{k+1})_i}{2^{n_{k+1}}\hat\phi_1(z_{k+1}\cdot\chi_{\N\setminus\{1,2,\ldots,m_k\}})}, & \text{if }m_k<i\leq m_{k+1}.
\end{cases}
\]
Then we have:
\[
\hat\phi_2(x_{k+1}\chi_{\{m_{k}+1,m_{k}+2,\ldots,m_{k+1}\}})=\frac{\hat\phi_2(z_{k+1}\chi_{\{m_{k}+1,m_{k}+2,\ldots,m_{k+1}\}})}{2^{n_{k+1}}\hat\phi_1(z_{k+1}\cdot\chi_{\N\setminus\{1,2,\ldots,m_k\}})}>2^{n_{k+1}}.
\]
Using Lemma \ref{lem-L}, find $L_{k+1}$ satisfying item (iii) and observe that all conditions are met.

Once the construction is completed, define $x=\bigcup_{k\in\N}x_k$. We need to show that $x\in\EXH(\phi_1)\setminus\FIN(\phi_2)$. 

The fact that $x\notin\FIN(\phi_2)$ follows from the observation that for each $k\in\N$ we have $\hat\phi_2(x)\geq\hat\phi_2(x_k\chi_{\{m_{k-1}+1,m_{k-1}+2,\ldots,m_k\}})>2^{n_{k+1}}$ (by item (vi)).

On the other hand, $x\in\EXH(\phi_1)$ follows from:
\[
\hat\phi_1(x\cdot\chi_{\N\setminus\{1,2,\ldots,m_k\}})\leq\sum_{i>k}\hat\phi_1(x_i\chi_{\{m_{i-1}+1,m_{i-1}+2,\ldots,m_i\}})=\sum_{i>k}\frac{1}{2^{n_i}}\leq\frac{1}{2^k}.
\]
\end{proof}

\subsection{Variations}

\begin{theorem}
\label{general-variations}
Let $\phi_1$ and $\phi_2$ be two non-pathological lsc submeasures and assume that $\phi_1(\{i\})>0$ and $\phi_2(\{i\})>0$ for every $i\in\N$. The following are equivalent:
\begin{itemize}
    \item[(a)] $\phi_2\preceq_m\phi_1$;
    \item[(b)] $\mEXH(\phi_1)\subseteq\mEXH(\phi_2)$; 
    \item[(c)] $\mFIN(\phi_1)\subseteq\mFIN(\phi_2)$;
    \item[(d)] $\mEXH(\phi_1)\subseteq\mFIN(\phi_2)$.
\end{itemize}
Moreover, assuming that $\phi_1$ and $\phi_2$ satisfy:
\begin{equation}\label{eq_monotonicity}
\forall_{j=1,2}\forall_{x,y\in\mathbb{R}^\N}\left(\left(\forall_{n\in\N}\sum_{i=1}^n |x_i|\leq\sum_{i=1}^n |y_i|\right)\implies \hat{\phi}_j(x)\leq\hat{\phi}_j(y)\right),
\end{equation}
the above conditions are also equivalent to the following one:
\begin{itemize}
    \item[(e)] $\BV(\phi_1)\subseteq\BV(\phi_2)$.
\end{itemize}
\end{theorem}

\begin{proof}
Firstly, we will show equivalence of items (a), (c) and (d). Secondly, we will show the implications (b)$\implies$(d) and (a)$\implies$(b). Lastly, we will deal with item (e) by showing (a)$\implies$(e) and (e)$\implies$(c).

(a)$\implies$(c): Straightforward, since $\hat\phi(x)=\lim_{n\to\infty}\hat\phi(x\cdot\chi_{\{1,2,\ldots,n\}})$ for every $x\in\R^\N$.

(c)$\implies$(d): This follows from the inclusion $\mEXH(\phi_1)\subseteq\mFIN(\phi_1)$ (see Proposition \ref{prop:mEXHsubsetmFIN}).

(d)$\implies$(a): Suppose that $\phi_2\not\preceq_m\phi_1$, i.e., for every $n\in\N$ there is a finite nonincreasing real sequence $z$ such that $\hat\phi_2(z)>2^{2n+1}\hat\phi_1(z)$.

We will recursively construct sequences $(n_k),(m_k),(L_k)\subseteq\N$ and $(x_k)\subseteq\bigcup_{n\in\N}\R^\N$ such that for each $k\in\N$:
\begin{itemize}
    \item[(i)] $n_1=1$ and $n_{k+1}>n_k$;
    \item[(ii)] $2^{2n_{k+1}}>L_k$;
    \item[(iii)] $\hat\phi_2(y)\leq L_k\hat\phi_1(y)$ for every $y\in\R^{m_k}$;
    \item[(iv)] $m_k$ is the length of $x_k$;
    \item[(v)] $\hat\phi_1(x_k\chi_{\{m_{k-1}+1,m_{k-1}+2,\ldots,m_k\}})=\frac{1}{2^{n_k}}$;
    \item[(vi)] $\hat\phi_2(x_k\chi_{\{m_{k-1}+1,m_{k-1}+2,\ldots,m_k\}})>2^{n_k}$;
    \item[(vii)] $\frac{1}{2^{n_{k+1}}}<\hat{\phi}_1(e_{m_k+1}x_k(m_k))$ (here $e_i\in\mathbb{R}^\N$ is the sequence having $1$ on $i$th coordinate and zeros on all other coordinates);
    \item[(viii)] $x_k\in\Mon$.
\end{itemize}
Note that items (i)-(vi) are exactly the same as in the proof of the implication (d)$\implies$(a) in Theorem \ref{general-sequences}. Hence, we will omit some details.

Start with $n_1=1$ (note that item (vii) is empty in the case of $k=1$), find some finite nonincreasing real sequence $z_1$ such that $\hat\phi_2(z_1)>2^{2n_1+1}\hat\phi_1(z_1)$ and put $x_1=\frac{z_1}{2^{n_1}\hat\phi_1(z_1)}$. Note that item (viii) is satisfied, since $z_1$ is nonincreasing. Let $m_1$ be the length of $x_1$. Then items (v) and (vi) are satisfied for the same reason as in the proof of Theorem \ref{general-sequences}. Moreover, using Lemma \ref{lem-L}, we find $L_1\in\N$ as in (iii).

If $n_i,m_i,L_i$ and $x_i$ for all $i\leq k$ are already defined, we can find $n_{k+1}$ such that items (i), (ii) and (vii) are met. There is some finite nonincreasing real sequence $z_{k+1}$ such that $\hat\phi_2(z_{k+1})>2^{2n_{k+1}+1}\hat\phi_1(z_{k+1})$. Let $m_{k+1}$ be the length of $z_{k+1}$. Similarly as in the proof of Theorem \ref{general-sequences}, we have $m_{k+1}>m_k$ and $\hat\phi_2(z_{k+1}\cdot\chi_{\N\setminus\{1,2,\ldots,m_k\}})>2^{2n_{k+1}}\hat\phi_1(z_{k+1}\cdot\chi_{\N\setminus\{1,2,\ldots,m_k\}})$. Put:
\[
(x_{k+1})_i=\begin{cases}
    (x_k)_i, & \text{if }i\leq m_k,\\
    \frac{(z_{k+1})_i}{2^{n_{k+1}}\hat\phi_1(z_{k+1}\cdot\chi_{\N\setminus\{1,2,\ldots,m_k\}})}, & \text{if }m_k<i\leq m_{k+1}.
\end{cases}
\]
Then items (iv), (v) and (vi) are met. We will show that (viii) is satisfied. Observe that $x_{k+1}\restriction\{1,\ldots,m_k\}$ is nonincreasing by (viii) applied to $k$ and $x_{k+1}\restriction\{m_k+1,m_k+2,\ldots\}$ is nonincreasing, since $z_{k+1}$ is. Thus, it suffices to check that $|x_{k+1}(m_k)|\geq|x_{k+1}(m_k+1)|$ and the latter follows from:
\begin{align*}
    |x_{k+1}(m_k+1)|\hat{\phi}_1(e_{m_k+1}) & =\hat{\phi}_1(|x_{k+1}(m_k+1)|e_{m_k+1})\cr
    & \leq\hat\phi_1(x_{k+1}\chi_{\{m_{k}+1,m_{k}+2,\ldots,m_{k+1}\}})=\frac{1}{2^{n_{k+1}}}\cr
    & <\hat{\phi}_1(e_{m_k+1}x_k(m_k))=|x_{k+1}(m_k)|\hat{\phi}_1(e_{m_k+1}).
\end{align*}
To end the recursion step, use Lemma \ref{lem-L} to find $L_{k+1}$ as in (iii).

Once the construction is completed, define $x=\bigcup_{k\in\N}x_k$. Then $x\in\Mon$ follows from (viii) and $x\in\EXH(\phi_1)\setminus\FIN(\phi_2)$ can be shown in the same way as in the proof of Theorem \ref{general-sequences}. 

(b)$\implies$(d): This follows from the inclusion $\mEXH(\phi_2)\subseteq\mFIN(\phi_2)$ (see Proposition \ref{prop:mEXHsubsetmFIN}).

(a)$\implies$(b): Assume that $\mEXH(\phi_1)\not\subseteq\mEXH(\phi_2)$ and take any $x\in \mEXH(\phi_1)\setminus \mEXH(\phi_2)$. Then $x\in\Mon\setminus\mEXH(\phi_2)$, thus $\mFIN(\phi_2)\subseteq c_0 $ (by Proposition~\ref{prop: mexh-basic}).
By the fact that $\mEXH(\phi_1)\subseteq\mFIN(\phi_1)$ (Proposition \ref{prop:mEXHsubsetmFIN}), we have  $\hat{\phi_1}(x)<\infty$, thus,
without losing generality, we may assume that $x$ is non-negative, non-increasing and  $\hat{\phi_1}(x)=1$.

There are two possible cases: either $x\not\in\mFIN(\phi_2)$ or $x\in\mFIN(\phi_2)$.

In the case that $x\not\in\mFIN(\phi_2)$, we have $\hat{\phi_2}(x)=\infty$ and $\hat{\phi_1}(x)=1$.  Since $\hat{\phi_2}(x)=\lim_{n\to\infty}\hat{\phi_2}(x\cdot\chi_{\{1,2,\ldots,n\}})$ and $\hat{\phi_1}\left(x\cdot\chi_{\{1,2,\ldots,n\}}\right)\leq\hat{\phi_1}\left(x\right)=1$ for all $n$, we obtain:
$$\lim_{n\to\infty}\frac{\hat{\phi_2}\left(x\cdot\chi_{\{1,2,\ldots,n\}}\right)}{\hat{\phi_1}\left(x\cdot\chi_{\{1,2,\ldots,n\}}\right)}\geq \lim_{n\to\infty}\hat{\phi_2}\left(x\cdot\chi_{\{1,2,\ldots,n\}}\right)=\infty, $$
hence $\phi_2\not\preceq_m\phi_1$.

In the case that $x\in\mFIN(\phi_2)$, we know that $x\in c_0$, as $\mFIN(\phi_2)\subseteq c_0 $. Take $\alpha>0$ such that $\lim_{n\to\infty} \hat{\phi_2}(x\cdot\chi_{\{n,n+1,\ldots\}})=\alpha$ (such $\alpha$ exists, as $\hat{\phi_2}(x\cdot\chi_{\{n,n+1,\ldots\}})$ is a non-increasing sequence in $[0,\hat{\phi_2}(x)]$ and $x\notin\mEXH(\phi_2)$). 
Define recursively the sequence $(n_k)$ in such way that for all $k\in\N$ we get $\hat{\phi_1}(x\cdot\chi_{\{n_k+1,n_k+2,\ldots\}})\leq 1/k$ and $$\frac{x(n_{k+1})}{x(n_k)}\leq\frac{1}{k}.$$ 
This is possible, because $\lim_{n\to\infty}x(n)=0$, $x$ is non-increasing and $\lim_{n\to\infty} \hat{\phi_1}(x\cdot\chi_{\{n,n+1,\ldots\}})=0$.

Now, for any $k\in\N$, we can define the sequence $y_k$ by $y_k(n)=x(n_{k+1})$ for $n\leq n_{k+1}$ and $y_k(n)=x(n)$ otherwise. Clearly, $y_k$ is non-increasing as $x$ is non-increasing. Moreover, taking $m_k$ such that $\hat{\phi_2}\left(x\cdot\chi_{\{n_{k+1}+1,n_{k+1}+2,\ldots,m_k\}}\right)\geq \alpha/2$, we may notice that:
$$\frac{\hat{\phi_2}\left(y_k\cdot\chi_{\{1,2,\ldots,m_k\}} \right) }{\hat{\phi_1}\left(y_k\cdot\chi_{\{1,2,\ldots,m_k\}} \right) }\geq \frac{\hat{\phi_2}\left(y_k\cdot\chi_{\{n_{k+1}+1,n_{k+1}+2,\ldots,m_k\}}\right) }{\hat{\phi_1}\left(y_k\cdot \chi_{\{1,2\ldots,n_k  \}}   \right) +\hat{\phi_1}\left(y_k\cdot \chi_{\{n_k+1,n_k+2,\ldots,m_k  \}}   \right)} \geq   $$
$$\geq \frac{\hat{\phi_2}\left(x\cdot\chi_{\{n_{k+1}+1,n_{k+1}+2,\ldots,m_k\}}\right) }{\frac{x(n_{k+1})}{x(n_k)}\hat{\phi_1}\left(x\cdot \chi_{\{1,2\ldots,n_k  \}}   \right) +\hat{\phi_1}\left(x\cdot \chi_{\{n_k+1,n_k+2,\ldots,m_k  \}}   \right)}\geq \frac{\alpha/2}{\frac{1}{k}+\frac{1}{k} }=\frac{k\alpha}{4}, $$ 
hence $\phi_2\not\preceq_m\phi_1$.

(a)$\implies$(e): Assume that $\phi_2\preceq_m\phi_1$, i.e. there is $M>0$ such that $\phi_2(y)\leq M\phi_1(y)$for every non-increasing finite sequence $y$. Let $x\in\BV(\phi_1)$ and fix any $\hat{J}=(\hat{J}_n)\in\mathcal{P}_I$. If $\phi_1(\N)<+\infty$, then $BV(\phi_1)=BV(\phi_2)=B(I)$ (by Proposition \ref{prop-basic-1}), so we may assume that $\phi_1(\N)=+\infty$. Then we have $x(\hat J_n)\in c_0$ and we can apply the procedure described below. Let $\pi:\N\to\N$ be such that $|x(\hat{J}_{\pi(1)})|=\sup\{|x(\hat{J}_n)|:n\in\N\}$ and $|x(\hat{J}_{\pi(k+1)})|=\sup\{|x(\hat{J}_n)|:n\in\N\setminus\{\pi(1),\ldots,\pi(k)\}\}$ for all $k\in\N$. Then $J^\star=(\hat{J}_{\pi(n)})\in\mathcal{P}_I$ and $x(J^\star)\in\Mon$. Hence, by the condition imposed on $\phi_2$, we have:
\[
\hat{\phi}_2(x(\hat{J}))\leq\hat{\phi}_2(x(J^\star))\leq M\hat{\phi}_1(x(J^\star))\leq M\sup_{J\in\mathcal{P}_I}\hat{\phi}_1(x(J)).
\]
Since $\hat{J}$ was arbitrary, we get that 
$$\sup_{J\in\mathcal{P}_I}\hat{\phi}_2(x(J))\leq M\sup_{J\in\mathcal{J}_I}\hat{\phi}_1(x(J)),$$ 
so $x\in\BV(\phi_2)$.

(e)$\implies$(c): We will show that the condition $\mFIN(\phi_1)\not\subseteq\mFIN(\phi_2)$ implies $\BV(\phi_1)\not\subseteq\BV(\phi_2)$. 

Assume first that $\mFIN(\phi_1)=\Mon$. Since $\mFIN(\phi_1)\not\subseteq\mFIN(\phi_2)$, $\mFIN(\phi_2)\neq\Mon$. Then Proposition \ref{prop-basic-1} gives us $\BV(\phi_1)=B(I)\not\subseteq\BV(\phi_2)$.

Assume now that $\mFIN(\phi_1)\not=\Mon$. Fix $x=(x_n)\in\mFIN(\phi_1)\setminus\mFIN(\phi_2)$. Let $f:[0,1]\to\mathbb{R}$ be a piecewise linear function such that $f(1)=0$, $f(0)=\sum_{n=1}^\infty (-1)^{n+1}|x_n|$ and $f(\frac{1}{2^k})=\sum_{n=1}^k (-1)^{n+1}|x_n|$ for all $k\in\N$. This function is well-defined as $x\in c_0$ (by Proposition \ref{prop-basic-1}).. 

Observe that $f\notin\BV(\phi_2)$, since for the sequence of intervals $\hat{J}=(\hat{J}_n)\in\mathcal{J}_I$ given by $\hat{J}_1=[\frac{1}{2},1]$ and $\hat{J}_{k+1}=[\frac{1}{2^{k+1}},\frac{1}{2^{k}}]$ for all $k\in\N$, we have $|f(\hat{J}_k)|=|x_k|$. Thus $\sup_{J\in\mathcal{J}_I}\hat{\phi}_2(f(J))\geq\hat{\phi}_2(f(\hat{J}))=\hat{\phi}_2(x)=\infty$. On the other hand, for each $J\in\mathcal{J}_I$ we have:
\[
\forall_{n\in\N}\sum_{i=1}^n |f(J_i)|\leq\sum_{i=1}^n |f(\hat{J}_i)|.
\]
This is actually a simple observation, which may either be proved directly or deduced from the general observation (see \cite[Proposition 1.1]{PW-Nachrichten}) stating that when selecting the ends of the intervals we should select {\em points of varying monotonicity} to get higher value of the sum.

Hence, by the assumption imposed on $\phi_1$, we have $\sup_{J\in\mathcal{J}_I}\hat{\phi}_1(f(J))=\hat{\phi}_1(f(\hat{J}))=\hat{\phi}_1(x)<\infty$.
\end{proof}

\section{Particular case I: simple density ideals}

In this Section we are interested in lsc submeasures of the form $\phi_g(C)=\sup_{n\in\N}\frac{|C\cap\{1,2,\ldots,n\}|}{g(n)}$, where $g:\N\to\N$ satisfies the following conditions:
\begin{itemize}
    \item[(a)] $(g(n))$ is nondecreasing;
    \item[(b)] $\lim_n g(n)=\infty$;
    \item[(c)] $\frac{n}{g(n)}$ does not tend to zero;
    \item[(d)] $\frac{n}{g(n)}$ is nondecresing.
\end{itemize}
Ideals of the form $\Exh(\phi_g)$ (for functions $g$ satisfying all the above requirements except the last one) have been extensively studied in \cite{simpleden3}, \cite{simpleden2} and \cite{simpleden1}. For that reason, we decided to write the last item separately despite the fact that (c) follows from (d). We will denote by $\mathcal{G}$ the set of all functions $g$ satisfying conditions (a)-(d). Note that in this case $\hat{\phi}_g(x)=\sup_{n\in\N}\frac{\sum_{i=1}^n|x_i|}{g(n)}$.

In 1974 Chanturia introduced the concept of the {\em modulus of variation of the bounded function} (see \cite{Cha74}), which for $x\colon I\to\R$ is given as a sequence 
\[
v(x,n) = \sup_{{\mathcal P}_n} \sum_{k=1}^n |x(I_k)|,
\]
where ${\mathcal P}_n$ denotes the set of all $n$-element families of nonoverlapping intervals of $I$. In the mentioned paper \cite{Cha74} the Author introduced set of functions $V[g]$, for a given sequence $g\colon \N\to \R$, as a family of those functions for which $v(x,n) = O(g(n))$. These classes are now called {\em Chanturia (or Chanturiya) classes} in literature. One of the statements (see Theorem 1 in the mentioned paper) was that the necessary and sufficient condition for a sequence to be $v(x,n)$ for some function $x$ is that it is nondecreasing and concave.   These classes were studied since then in many papers, mainly in relation to a convergence of Fourier series and relations to other families of functions of bounded variation (see especially the relation between Chanturia classes and Waterman spaces given by Avdispahi\'c in \cite{Avdispahic}).

The sequences defining Chanturia classes are defined as real valued sequences but as we can see we may redefine any real-valued sequence $h(n)$ to a sequence $g(n) = \lceil h(n)\rceil$, which has values in natural numbers and the same asymptotics as $n\to+\infty$. One more important observation for nondecreasing and concave sequences $h(n)$ as considered in a context of Chanturia classes is that they are such that $\frac{n}{h(n)}$ are nondecreasing (as required by the definition of a family $\mathcal{G}$ above). Without the loss of generality we may assume that $h(0)=0$ and then from being concave we may deduce that for each $k\in\{1,...,n-1\}$ we have
\[
h(k) = h\left(\frac{k}{n}n + \left(1-\frac{k}{n}\right)0\right) \geq \frac{k}{n} h(n) + \left(1-\frac{k}{n}\right)h(0) = \frac{k}{n}h(n),
\]
which gives
\[
\frac{h(k)}{k} \geq \frac{h(n)}{n},
\]
as desired.

\begin{theorem}
The following are equivalent for every $g,h\in\mathcal{G}$:
\begin{itemize}
\item[(a)] $g(n) = O(h(n))$, i.e., there is $\eta>0$ such that $\frac{g(n)}{h(n)}\leq\eta$ for all $n\in N$;
\item[(b)] $\phi_h\preceq\phi_g$;
\item[(c)] $\phi_h\preceq_m\phi_g$;
\item[(d)] $\Fin(\phi_g)\subseteq\Fin(\phi_h)$;
\item[(e)] $\Exh(\phi_g)\subseteq\Exh(\phi_h)$;
\item[(f)] $\Exh(\phi_g)\subseteq\Fin(\phi_h)$;
\item[(g)] $\FIN(\phi_g)\subseteq\FIN(\phi_h)$;
\item[(h)] $\EXH(\phi_g)\subseteq\EXH(\phi_h)$;
\item[(i)] $\EXH(\phi_g)\subseteq\FIN(\phi_h)$;
\item[(j)] $\mFIN(\phi_g)\subseteq\mFIN(\phi_h)$;
\item[(k)] $\mEXH(\phi_g)\subseteq\mEXH(\phi_h)$;
\item[(l)] $\mEXH(\phi_g)\subseteq\mFIN(\phi_h)$;
\item[(m)] $\BV(\phi_g)\subseteq\BV(\phi_h)$.
\end{itemize}
\end{theorem}

\begin{proof}
(a)$\implies$(b): We claim that $\phi_h(x)\leq\eta\phi_g(x)$ for every $x\in\mathbb{R}^\N$. Indeed, for every $n\in\N$ we have:
\[
\frac{\sum_{i=1}^n |x_i|}{h(n)}\leq\eta\frac{\sum_{i=1}^n |x_i|}{g(n)}.
\]

Items (b), (g), (h) and (i) are equivalent thanks to Theorem \ref{general-sequences}.

The (b)$\implies$(c) is obvious.

%Proposition \ref{general-variations-EXH} gives us (b)$\implies$(k) and (k)$\implies$(c).

Items (c), (j), (k), (l) and (m) are equivalent thanks to Theorem \ref{general-variations}, since for any $g\in\mathcal{G}$ and any $x,y\in\mathbb{R}^\N$ such that $\sum_{i=1}^n |x_i|\leq \sum_{i=1}^n |y_i|$ for all $n\in\N$, it is easy to see that $\phi_g(x)\leq\phi_g(y)$.

By Proposition \ref{general-ideals}, (b)$\implies$(d), (b)$\implies$(e), (d)$\implies$(f) and (e)$\implies$(f).

Therefore, we only need to show (f)$\implies$(a) and (c)$\implies$(a).

(c)$\implies$(a): Assume that (a) does not hold. We need to show that (c) does not hold, i.e., for every $M>0$ there is a finite nonincreasing sequence $x$ such that $\phi_h(x)>M\phi_g(x)$.

Fix $M>0$. Since (a) does not hold, there is $n\in\N$ such that $\frac{g(n)}{h(n)}>M$. Define $x=(x_i)\in\mathbb{R}^\N$ by:
\[
x_i=\begin{cases}
    \frac{g(n)}{n}, & \text{if }i\leq n,\\
    0, & \text{otherwise}.
\end{cases}
\]
Then $x$ is nonincreasing and 
\[
\phi_h(x)\geq\frac{\sum_{i=1}^n |x_i|}{h(n)}=\frac{g(n)}{h(n)}>M.
\]
Now we will show that $\phi_g(x)\leq 1$, which will finish the proof. Fix $m\in\N$. There are three possibilities:
\begin{itemize}
    \item If $m=n$, then $\frac{\sum_{i=1}^n |x_i|}{g(m)}=1$.
    \item If $m>n$, then $\frac{\sum_{i=1}^n |x_i|}{g(m)}=\frac{g(n)}{g(m)}\leq 1$, since $g$ is nondecreasing.
    \item If $m<n$, then 
    \[
    \frac{\sum_{i=1}^n |x_i|}{g(m)}=\frac{\frac{g(n)}{n}}{\frac{g(m)}{m}}\leq 1,
    \] 
    since $\frac{n}{g(n)}$ is nondecresing.
\end{itemize}

(f)$\implies$(a): Assume that (a) does not hold. We will use Proposition \ref{general-ideals}(c) to show that (f) does not hold. 

Fix any $k\in\N$. We are looking for $F\in\Fin$ such that $\phi_g(F)\leq\frac{1}{2^k}$ and $\phi_h(F)\geq 2^k$. Let $\delta>1$ be such that $\frac{1}{g(1)}\geq\frac{1}{\delta}$. Note that $\frac{i}{g(i)}\geq\frac{1}{\delta}$ for all $i\in\N$ (by the fact that $\frac{i}{g(i)}$ is nondecresing).

Since (a) does not hold, there is $n\in\N$ such that $\frac{g(n)}{h(n)}>2^k 2^{k+1}\delta$. Actually, there are infinitely many such $n$ (given one such $n$ we can always find $n'\in\N$ with $\frac{g(n')}{h(n')}>\frac{g(n)}{h(n)}>2^k 2^{k+1}\delta$). Thus, without loss of generality we may assume that $n$ is big enough to guarantee that $\frac{1}{g(n)}<\frac{1}{2^k\delta}<\frac{1}{2^k}$ (since $\lim_n g(n)=\infty$).

Find $j\in\N$ such that $2^{j+1}\delta\geq g(n)>2^j\delta$ and note that $j\geq k$ (as $\frac{1}{g(n)}<\frac{1}{2^k\delta}$). Moreover, since $\frac{n}{g(n)}\geq\frac{1}{\delta}$, we have $n\geq\frac{g(n)}{\delta}>2^j$. Hence, we can find $a\leq n$ such that $n-a=2^{j-k}$. 

Define $F=\{n-a+1,n-a+2,\ldots,n\}$. Then \[
\phi_h(F)\geq \frac{n-a}{h(n)}>\frac{n-a}{g(n)}2^k 2^{k+1}\delta>\frac{2^{j-k}}{2^{j+1}\delta}2^k 2^{k+1}\delta=2^k.
\]
In order to finish the proof, we need to show that 
$$\phi_g(F)=\sup_{m\in\N}\frac{|F\cap\{1,\ldots,m\}|}{g(m)}\leq\frac{1}{2^k}.$$ 
Fix $m\in\N$. There are four possibilities:
\begin{itemize}
    \item If $m=n$, then $\frac{|F\cap\{1,\ldots,m\}|}{g(m)}=\frac{n-a}{g(n)}<\frac{2^{j-k}}{2^j\delta}<\frac{1}{2^k}$.
    \item If $m>n$, then $\frac{|F\cap\{1,\ldots,m\}|}{g(m)}=\frac{n-a}{g(m)}\leq \frac{n-a}{g(n)}<\frac{1}{2^k}$, since $g$ is nondecreasing.
    \item If $m\leq a$, then $\frac{|F\cap\{1,\ldots,m\}|}{g(m)}=0$.
    \item If $a<m<n$, then $\frac{|F\cap\{1,\ldots,m\}|}{g(m)}=\frac{m-a}{g(m)}=\frac{m}{g(m)}-\frac{a}{g(m)}\leq\frac{n}{g(n)}-\frac{a}{g(n)}=\frac{n-a}{g(n)}<\frac{1}{2^k}$, since $\frac{n}{g(n)}$ and $g$ are nondecresing.
\end{itemize}
\end{proof}

\section{Particular case II: summable ideals and \texorpdfstring{$\Lambda BV$}{} spaces}

In this Section we are interested in lsc submeasures of the form $\phi_A(C)=\sum_{n\in C}a_n$, where $A=(a_n)$ is a Waterman sequence (i.e., $A$ is a nonincreasing sequence of positive real numbers such that $\sum_{n\in\N}a_n=\infty$). Note that in this case $\phi$ is actually a measure and $\hat{\phi}_A(x)=\sum_{n\in\N}a_i|x_i|$. In such case:
\begin{itemize}
    \item $\Fin(\phi_A)=\Exh(\phi_A)$ is the summable ideal $\I_A$ given by the sequence $A$ (see Definition \ref{def-summable} and \cite[Example 1.2.3]{Farah});
    \item $\FIN(\phi_A)=\EXH(\phi_A)$ and $\mFIN(\phi_A)=\mEXH(\phi_A)$ (by the previous item, Proposition \ref{Exh=Fin} and Remark \ref{Exh=Fin-rem});
    \item $\BV(\phi_A)$ is the space $\ABV$ of functions of bounded $A$-variation.
\end{itemize} 

We will need the following result.

\begin{theorem}[{\cite[Theorem~4.1]{Chinczycy}}]
Assume we have two Waterman sequences $A=(a_n)_{n\in\N}$ and $B=(b_n)_{n\in\N}$. The following statements are equivalent:
\begin{enumerate}
\item[(a)] $\I_A\leq_K\I_B$;
\item[(b)] There exist positive numbers $m,M$ and a partition of $\N$ into consecutive nonempty  finite intervals $(I_n)_{n\in\N}$ such that 
$$ ma_n\leq\sum_{i\in I_n}b_i \leq Ma_n$$ 
for every $n\in\N$;
\item[(c)] There exists $M>0$ such that 
$$\sum_{i=1}^k b_i\geq M\sum_{i=1}^l a_i \Rightarrow b_k\leq Ma_l$$
for all $k,l\in\N$.
\end{enumerate}    
\end{theorem}

\begin{theorem}\label{th:inclusions}
Assume we have two  Waterman sequences $A=(a_n)_{n\in\N}$ and $B=(b_n)_{n\in\N}$. The following statements are equivalent:
\begin{enumerate}
\item[(a)] $\sum_{i=1}^n b_i = O(\sum_{i=1}^n a_i)$, i.e., there is $\eta>0$ such that 
\[
\frac{\sum_{i=1}^n b_i}{\sum_{i=1}^n a_i}\leq\eta
\]
for all $n\in N$;
\item[(b)] $\I_A\leq_K\I_B$;
\item[(c)] $\phi_B\preceq_m\phi_A$;
\item[(d)] $\ABV\subseteq\BBV$;
\item[(e)] $\mEXH(\phi_A)\subseteq\mEXH(\phi_B)$.
\item[(f)] $\mFIN(\phi_A)\subseteq\mFIN(\phi_B)$;
\item[(g)] $\mEXH(\phi_A)\subseteq\mFIN(\phi_B)$;
\end{enumerate}
\end{theorem}

\begin{proof}
The equivalence of items (c)-(g) follows from Theorem \ref{general-variations}, since given any Waterman sequence $A=(a_n)_{n\in\N}$ and any $x,y\in\mathbb{R}^\N$ such that $\sum_{i=1}^n |x_i|\leq \sum_{i=1}^n |y_i|$ for all $n\in\N$, using the fact that $(a_n)$ is nonincreasing we have:
\[
\sum_{k=1}^n a_k|x_k|=\sum_{k=1}^{n-1}\left(\sum_{i=1}^k |x_i|\right)(a_k-a_{k+1})+a_n\sum_{i=1}^n|x_i|\leq 
\]
\[
\leq\sum_{k=1}^{n-1}\left(\sum_{i=1}^k |y_i|\right)(a_k-a_{k+1})+a_n\sum_{i=1}^n|y_i|=\sum_{k=1}^n a_k|y_k|.
\]
Hence, $\phi_A(x)\leq\phi_A(y)$.

The equivalence of items (a) and (d) is proved in \cite[Theorem 3]{PerlmanWaterman}, so it remains to show that (a) and (b) are equivalent.

$(b)\Rightarrow (a)$: By the previous theorem, we can find $M>0$ and a partition of $\N$ into consecutive nonempty intervals $(I_n)_{n\in\N}$ such that for all $n\in\N$ we have
$$\frac{\sum_{i\in I_n}b_i}{a_n}\leq M ,$$ thus 
$$\frac{\sum_{i=1}^{\max I_n} b_i}{\sum_{i=1}^n a_i}\leq M. $$

Let us put $\eta=M$ and take $n\in\N$. Then there exists $j\in\N$ such that $n\in  I_j$. It follows that $n\geq j$, thus
$$\frac{\sum_{i=1}^n b_i}{\sum_{i=1}^n a_i}\leq \frac{\sum_{i=1}^{\max I_j} b_i}{\sum_{i=1}^j a_i}\leq M=\eta. $$

$(a)\Rightarrow (b)$: Suppose that $\I_A\not\leq_K\I_B$ and fix $\eta>0$. Let $M=2\eta$. By the previous theorem, we can find $k,l\in\N$ such that $\sum_{i=1}^k b_i\geq M\sum_{i=1}^l a_i$  and  $b_k> M a_l$. We have two cases. 

If $k\leq l$ then 
$$\frac{\sum_{i=1}^k b_i}{\sum_{i=1}^k a_i}\geq \frac{\sum_{i=1}^k b_i}{\sum_{i=1}^l a_i}\geq  M> \eta. $$

If $k>l$ then
$$\sum_{i=l+1}^k a_i\leq \sum_{i=l+1}^k a_l< \sum_{i=l+1}^k \frac{b_k}{M}\leq \sum_{i=l+1}^k \frac{b_i}{M}\leq \sum_{i=1}^k\frac{b_i}{M},    $$
thus 
$$\sum_{i=1}^k a_i=\sum_{i=1}^l a_i+\sum_{i=l+1}^k a_i<\frac{\sum_{i=1}^k b_i}{M}+\sum_{i=1}^k\frac{b_i}{M}=\frac{2\sum_{i=1}^k{b_i}}{M},   $$
hence 
$$\frac{\sum_{i=1}^k b_i}{\sum_{i=1}^k a_i} >  \frac{M}{2}=\eta. $$

Therefore, in both cases condition (a) does not hold. 
\end{proof}

\section{Comparison of two particular cases}

The next theorem is an interesting comment to the paper \cite{Avdispahic} giving different inclusions between Chanturia classes and Waterman spaces: this is a general proof that many of the inclusions given there are strict (these, for which the  assumption on monotonicity of  $(g(n+1)-g(n))$ is satisfied).

\begin{theorem}
Let $g\in\mathcal{G}$ be such that $(g(n+1)-g(n))$ monotonically tends to $0$. Then $\BV(\phi_g)$ is not equal to any $\ABV$.
\end{theorem}
\begin{proof}
Let $A=(a_n)$ be a Waterman sequence. We will obtain the thesis by showing that $\mFIN(\phi_g)\not=\mFIN(\phi_A)$ and applying Theorem~\ref{general-variations}.

Define the sequence $x=(x_n)$ by the formula $\sum_{i=1}^{n} x_n= g(n)$. Clearly, $x$ is nonincreasing, tends to $0$ and belongs to $\mFIN(\phi_g)$ as $\hat{\phi}_g(x)=1$. We have two cases:

In the first case we assume that $\sum_{n=1}^{\infty}a_n x_n=\infty$. Then $x\in \mFIN(\phi_g)\setminus\mFIN(\phi_A)$.

In the second case we assume that $\sum_{n=1}^{\infty}a_n x_n<\infty$. Notice that then for every $k\in\N$ we have  $\sum_{n=1}^{\infty}a_n(k\cdot x_n)=k\cdot \sum_{n=1}^{\infty}a_n x_n<\infty$. We will show that $\mFIN(\phi_A)\setminus\mFIN(\phi_g)\not=\emptyset$.

We will now define recursively sequences $(n_i)$ and $(m_i)$ of natural numbers. We put as $n_1$ the smallest $N$ such that $2\cdot \sum_{n>N} a_n x_n<\frac{1}{4}$. Next, we put as $m_1$ the smallest $N> n_1$ such that $2 x_N\leq x_{n_1}$. We can find such because $\lim_{n\to\infty} x_n=0$. 

Now, suppose we have already defined $n_i$ and $m_i$ for some $i\in\N$. 
Then we define $n_{i+1}$ as the smallest $N>m_i$ such that 
$$(i+2) \sum_{n>N} a_n x_n<\frac{1}{2^{i+2}} \text{ and }(i+1)\sum_{n=m_i}^N x_n\geq \frac{i+1}{2}g(N).$$ 
We can find such $N$, since  $\lim_{j\to\infty} \sum_{n>j} a_n x_n=0$ and  $(i+1)\sum_{n=1}^j x_n=(i+1) g(j)$, which tends to infinity, thus 
$$\lim_{j\to\infty} \frac{(i+1) \sum_{n=m_i}^j x_n}{g(j)}=i+1.$$
Next, we define $m_{i+1}$ as the smallest $N>n_{i+1}$ such that $(i+2)x_N\leq (i+1) x_{n_{i+1}}$.

We can now proceed to defining the sequence $y=(y_n)$ such that $y\in \mFIN(\phi_A)\setminus\mFIN(\phi_g)$, which will end the proof. We put

$$y_n=\begin{cases}
    x_n, & \text{if }n\leq n_1,\\
    (i+1) x_{m_i}, & \text{if }n_i<n<m_i \text{ for some }i,\\
    (i+1)x_n, & \text{if }m_i\leq n\leq n_{i+1} \text{ for some }i.
\end{cases}
$$
First, notice that $y$ is nonincreasing as $(x_n)$ is nonincreasing and  $(i+1)x_{m_i}\leq i\cdot x_{n_i}$ for every $i\in\N$. Next, observe that $y\not\in\mFIN(\phi_g)$ as 
$$\hat\phi_g(y)\geq \frac{\sum_{n=m_i}^{n_{i+1}} (i+1)x_n}{g(n_{i+1})}\geq \frac{i+1}{2} $$
for every $i\in\N$. 
Finally, we obtain $y\in\mFIN(\phi_A)$ by the fact that
$$\sum_{n=1}^{\infty} a_n y_n \leq\sum_{n\leq n_1}a_n x_n + \sum_{i=1}^{\infty} \sum_{n=n_i+1}^{n_{i+1}} a_n(i+1)x_n\leq $$
$$\leq \sum_{n\leq n_1}a_n x_n + \sum_{i=1}^{\infty} \sum_{n>n_i}(i+1)a_n x_n <  \sum_{n\leq n_1}a_n x_n + \sum_{i=1}^{\infty} \frac{1}{2^{i+1}} =$$
$$=\sum_{n\leq n_1}a_n x_n +\frac{1}{2}<\infty.$$
\end{proof}

\begin{example}
The sequence $g(n)=\sqrt{n}$ belongs to $\mathcal{G}$ and is such that $g(n+1)-g(n)$ monotonically tends to $0$. Therefore,  $\BV(\phi_g)$ is not equal to any $\ABV$.
\end{example}

\bibliographystyle{amsplain}
\bibliography{references.bib}

\end{document}